\documentclass[12pt]{article}
\usepackage{latexsym}
\usepackage{amssymb}
\usepackage{amsthm}
\begin{document}
\newtheorem{theorem}{Theorem}[section]
\newtheorem{lemma}[theorem]{Lemma}
\newtheorem{proposition}[theorem]{Proposition}
\newtheorem{corollary}[theorem]{Corollary}
\newtheorem*{definition}{Definition}
\newtheorem*{question}{Question}
\newtheorem{conjecture}{Conjecture}
\newtheorem*{thm}{Theorem}
\newcommand{\F}{\ensuremath{\mathbb F}}
\newcommand{\N}{\mathcal N}
\newcommand{\R}{\mathcal R}
\newcommand{\Z}{\mathbb Z}

\title{Groups having all elements off a normal subgroup with prime power order}
\author{Mark L. Lewis}
\date \
\maketitle

\begin{abstract}
We consider a finite group $G$ with a normal subgroup $N$ so that all elements of $G \setminus N$ have prime power order.  We prove that if there is a prime $p$ so that all the elements in $G \setminus N$ have $p$-power order, then either $G$ is a $p$-group or $G = PN$ where $P$ is a Sylow $p$-subgroup and $(G,P,P \cap N)$ is a Frobenius-Wielandt triple.  We also prove that if all the elements of $G \setminus N$ have prime power orders and the orders are divisible by two primes $p$ and $q$, then $G$ is a $\{ p, q \}$-group and $G/N$ is either a Frobnius group or a $2$-Frobenius group.   If all the elements of $G \setminus N$ have prime power orders and the orders are divisible by at least three primes, then all elements of $G$ have prime power order and $G/N$ is nonsolvable.\\[1ex]
{\it Keywords:} Prime power order, Frobenius groups, Frobenius-Wielandt groups\\[1ex]
{\it 2020 Mathematics Subject Classification:20D99}\\[1ex]
\end{abstract}

\section{Introduction}

In this paper, unless otherwise noted, all groups are finite.  Our goal is to characterize the groups that have a proper normal subgroup where every element outside of the normal subgroup has prime power order.  To understand these groups, it is useful to first understand the groups where all elements have prime power order.  Note that in our question, this is the case where the normal subgroup is trivial.

The question of groups where all elements have prime power order was first addressed by Higman in \cite{Higman2} where he determined the solvable groups with this property.  Suzuki in \cite{suzuki} found the simple groups with this property, and then Brandl in \cite{brandl} completed the classification of these groups (with one omission).  We will list this classification and include more background and discussion on these groups in Section \ref{sec:eppo}. 

We first address the question of a group $G$ with a normal subgroup $N$ and a prime $p$ so that every element in $G \setminus N$ has $p$-power order.  Obviously, if $G$ is a $p$-group, then every normal subgroup $N$ will have this property.  It turns out to obtain the non $p$-groups with this property, one needs to look at a generalization of Frobenius groups that Wielandt studied.  We will introduce Frobenius-Wielandt triples in Section \ref{sec:frob}.  We will prove the following:

\begin{theorem} \label{main one}
Let $G$ be a group, let $N$ be a normal subgroup of $G$, and let $p$ be a prime.  Then all elements of $G \setminus N$ have $p$-power order if and only if either (1) $G$ is a $p$-group or (2) $P$ is a Sylow $p$-subgroup of $G$ so that $G = NP$ and $(G,P,P \cap N)$ is a Frobenius-Wielandt triple.
\end{theorem}

The next obvious question is what can occur when $G \setminus N$ contains all elements that have prime power but not for the same prime.  It turns out that the answer depends whether there are two primes or more than two primes.   Recall that $G$ is a $2$-Frobenius group if there exist normal subgroups $K < L < G$ so that $G/K$ and $L$ are Frobenius groups with Frobenius kernels $L/K$ and $K$, respectively.  We first address the case with two primes.

\begin{theorem} \label{main two}
Let $G$ be a group and let $N$ be a normal subgroup of $G$.  Suppose that all elements of $G \setminus N$ have prime power order and that two distinct primes $p$ and $q$ divide the orders of such elements.  Then the following are true: $G$ is a $\{ p, q \}$-group for distinct primes $p$ and $q$ and either $G/N$ is either a Frobenius group or a $2$-Frobenius group.
\end{theorem}

In fact, we will give a necessary a sufficient condition for this case in terms of multiple Frobenius-Wielandt triples.  Using the results for the case with two primes we are able to obtain the following result for three or more primes.

\begin{theorem} \label{main three}
Let $G$ be a group and let $N$ be a normal subgroup of $G$.  Suppose that all elements of $G \setminus N$ have prime power orders and that at least three distinct primes divide the orders of such elements.  Then all elements in $G$ have prime power order.  In fact, $G$ is one of the groups listed in Theorem \ref{nonsolvable}.
\end{theorem}

The question of groups with a normal subgroup having all elements outside of the normal subgroup with prime power order arose in our Theorem 3.2 of preprint with Bianchi, Camina, and Pacifici \cite{ourpre}.

\section{Groups with prime power order}\label{sec:eppo}

As we stated in the Introduction, Higman introduced the study of groups where all elements have prime power order in \cite{Higman2}.  Suzuki continued by finding the simple groups with this property in \cite{suzuki}, and Brandl completed the classification for all finite groups.  A couple of more recent papers have also given the complete classification of these groups apparently unaware of Brandl's paper (\cite{hein} and \cite{ShiYang2}).  In the literature, these groups have been called CP-groups and EPPO-groups.  We note that the classification has been further refined to groups where all elements have prime order in \cite{Cheng93} and \cite{Deaconescu89}.  Furthermore, we obtain motivation from the fact that the groups $G$ having a normal subgroup $N$ so that every element in $G \setminus N$ has prime order has been classified by Qian in \cite{Qian2005}.  (Unfortunately, \cite{Qian2005} is in Chinese and hence, we do not know of an English version of this proof.  Also, Qian's classification does not seem  to follow immediately follow from ours.  The condition of the elements having prime order seems to be significantly stronger than prime power order, so we have not included this classification.)  We should also mention that classifying normal subgroups with all elements outside having prime order for a given prime gets involved with the Hughes' subgroup and the associated questions.  See  \cite{HVL} and the references therein for one perspective on that problem.

Using modern tchniques, it is not difficult to replicate Higman's result.  In particular, the {\it Grueneberg-Kegel graph} (GK-graph) of a group $G$ is the graph whose vertex set is the set of prime divisors of $|G|$.  There is an edge between primes $p$ and $q$ if there is an element $g \in G$ so that $pq$ divides the order of $g$.  (Often, this graph has been called the prime graph, but since the term prime graph has been applied to more than one graph associated with the group, we prefer the name Grueneberg-Kegel since it is unambiguous.)  Recall that an empty graph is a graph no edges.  It is not difficult to see that all elements in a group $G$ have prime power order if and only if the GK-graph of $G$ is an empty graph.  Also, a group $G$ is a $p$-group for some prime $p$ if and only if the GK-graph consists of a single vertex.  Now, when $G$ is a group having prime power order that is not a $p$-group, then each vertex will be a connected component of the GK-graph and the GK-graph will disconnected.  Williams shows in \cite{will} that a solvable group with a disconnected GK-graph will be either a Frobenius group or $2$-Frobenius group and hence, have two connected components.  In our case, that implies that $G$ will be a $\{ p, q\}$-group for distinct primes $p$ and $q$.  We will remind the read of Frobenius and $2$-Frobenius groups in Section \ref{sec:frob}.  For reference in the sequel, we write out explicitly the classification of solvable groups with all elements having prime power order.

\begin{theorem}[Higman]\label{solvable}
Let $G$ be a solvable group.  Then every element of $G$ has prime power order if and only if one of the following occurs:
\begin{enumerate}
	\item $G$ is a $p$-group for some prime $p$.
	\item There exist distinct primes $p$ and $q$ so that $G$ is a $\{p,q\}$-group and either $G$ is a Frobenius group or $G$ is a $2$-Frobenius group.
\end{enumerate} 
\end{theorem}

We next also explicity write down the classification of nonsolvable groups with all elements have prime power order.  We note that Brandl misses the group $M_{10}$ (this is the nonsplit extension of ${\rm PSL}_2 (9)$ by $Z_2$ which occurs as a point stabilizer in $M_{11}$) in \cite{brandl}, but this group is mentioned in \cite{hein} and \cite{ShiYang2}.

\begin{theorem}[Brandl]\label{nonsolvable}
Let $G$ be a solvable group.  Then every element of $G$ has prime power order if and only if one of the following occurs:
\begin{enumerate}
	\item $G$ is isomorphic to ${\rm PSL}_2 (7)$, ${\rm PSL}_2 (9)$, ${\rm PSL}_2 (17)$, ${\rm PSL}_3 (4)$, or $M_{10}$.
	\item $G$ has a normal subgroup $N$ so that $G/N$ is isomorphic to one of ${\rm PSL}_2 (4)$, ${\rm PSL}_2 (8)$, ${\rm Sz} (8)$, or ${\rm Sz} (32)$ and either $N = 1$ or $N$ is a nontrivial, elementary abelian $2$-group that is isomorphic to a direct sum of natural modules for $G/N$.
\end{enumerate}	
\end{theorem}

\section{Frobenius-Wielandt triples} \label{sec:frob}

The study of Frobenius groups is one of the more celebrated topics in group theory.  Recall that a proper, nontrivial subgroup $H$ of a group $G$ is called a {\it Frobenius complement} if $H \cap H^g = 1$ for all $g \in G \setminus H$.  A group $G$ is called a {\it Frobenius group} if it contains a Frobenius complement $H$.  Frobenius' theorem states that if $N = G \setminus \cup_{g \in G} (H \setminus 1)^g$, then $N$ is a normal subgroup of $G$.  In addition $G = H N$ and $H \cap N = 1$.  The subgroup $N$ is called the {\it Frobenius kernel} of $G$.  It is known that $(|N|,|H|) = 1$ and that the Sylow subgroups of $H$ are either cyclic or generalized quaternion groups.  Also, Thompson proved that $N$ is nilpotent.  For a full treatment of Frobenius groups, we suggest that the reader refer to \cite{hup} or Chapter 6 of \cite{grouptxt}.

In \cite{wielandt}, Wielandt studies groups $G$ that have a proper, nontrivial subgroup $H$ and a normal subgroup $L$ of $H$ that is proper in $H$ for which $H \cap H^g \le L$ for all $g \in G \setminus H$.  In this situation, Epuelas in \cite{espuelas} says that $(G,H,L)$ is a Frobenius-Wielandt group.  Following \cite{burkett}, we say that $(G,H,L)$ is a {\it Frobenius-Wielandt triple}, if $L$ is normal in $H$ and $H \cap H^g \le L$ for all $g \in G \setminus H$.  Observe that $H$ is a Frobenius complement if $L = 1$; so this a generalization of Frobenius complements.  In \cite{wielandt}, Wielandt proves that $H$ and $L$ determine a unique normal subgroup $N$ so that $G = NH$ and $N \cap H = L$.  In fact, $N = G \setminus \cup_{g \in G}(H \setminus L)^g$.  The subgroup $N$ is called the {\it Frobenius-Wielandt kernel}.  It is not difficult to see when $L = 1$ that $N$ is the usual Frobenius kernel.  Wielandt also proved in that paper that $(|G:H|,|H:L|) = 1$.  We note that Section 4 of \cite{KnSch} gives a different proof Wielandt's results.

It is a natural question to ask how closely related to Frobenius groups are Frobenius-Wielandt triples.  This is the question addressed by Espuelas in \cite{espuelas}.  In that paper, among other things he proves that if $H$ splits over $N$ and $|H/N|$ is even, then $H/N$ is isormorphic to a Frobenius complement and if $|N|$ is odd and $q$ is a prime divisor of $|N|$ so that a Sylow $q$-subgroup of $N$ is abelian and is complemented in a Sylow $q$-subgroup of $H$, then the Sylow $q$-subgroups of $H/N$ are cyclic.  On the other hand, Scopolla in \cite{scopolla} proves that if $P$ is any $p$-group, then there is a group $G$ with a Sylow $p$-subgroup $Q$ and subgroup $L$ normal in $Q$ so that $(G,Q,L)$ is a Frobenius-Wielandt triple and $Q/L \cong P$.

We now prove some results regarding Frobenius-Wielandt triples.  

\begin{lemma} \label{frob quo}
If $(G,H,L)$ is a Frobenius-Wielandt triple, then either $N_G (L) = H$ or $N_G (L)/L$ is a Frobenius group.
\end{lemma}

\begin{proof}
Obviously, we have $H \le N_G (L)$.  Suppose $H < N_G (L)$.  Thus, we know that $H/L \cap H^x/L = (H \cap H^x)/L \le L/L$ for all $x \in N_G (L) \setminus H$.  It follows that $H/L$ is a Frobenius complement in $N_G (L)/L$.  We conclude that $N_G (L)/L$ is a Frobenius group.
\end{proof}

\begin{corollary}
If $(G,H,L)$ is a Frobenius-Wielandt triple and $L$ is normal in $G$, then $G/L$ is a Frobenius group.	
\end{corollary}

\begin{lemma} \label{conjugacy}
Let $N$ be a normal subgroup of a group $G$.  Suppose $H$ is a subgroup of $G$ so that $G = HN$.  Then every element of $G \setminus N$ is conjugate to an element in $H$ if and only if $(G,H,H \cap N)$ is a Frobenius-Wielandt triple.
\end{lemma}

\begin{proof}
Suppose first that $(G,H,H \cap N)$ is a Frobenius-Wielandt triple.  Since $G = NH$ and $N = G \setminus \cup_{x \in G}(H \setminus H \cap N)^x$, we see that $G \setminus N = \cup_{x \in G} (H \setminus H \cap N)^x$.  Thus, every element of $G \setminus N$ is conjugate to an element in $H$.  Conversely, suppose every element in $G \setminus N$ is conjugate to an element in $H$.  This implies that $G \setminus N = \cup_{x \in G}(H \setminus N)^x$.  Note that the number of conjugates of $H \setminus N$ is $|G:N_G (H)| \ge |G:H|$.  We have 
\begin{eqnarray*}
	|G \setminus N | &=& |\cup_{x \in G} (H \setminus N)^x| \le \sum_{x \in G} |H \setminus N| \\
	&\le& |G:N_G (H)| (|H| - |H \cap N|) \\
	&\le& |G:H| (|H| - |H \cap N| ) = |G:H||H| - |G:H||H \cap N| \\
	&= &|G| - |N:H \cap N||H \cap N| = |G| - |N| = |G \setminus N|.
\end{eqnarray*}

Thus, we must have equality throughout the equation.  This implies that $N_G (H) = H$ and $H \setminus H \cap N$ and $(H \setminus H \cap N)^x$ are disjoint when $x \not\in H$, so $H \cap H^x \subseteq H \cap N$ for all $x \in G \setminus H$.  Thus, $(G,H,H \cap N)$ is a Frobenius-Wielandt triple.
\end{proof}

We next show that quotients that are Frobenius groups yield Frobenius-Weilandt triples.

\begin{corollary}
Let $N$ be a normal subgroup of a group $G$.  If $G/N$ is a Frobenius group with Frobenius complement $H/N$, then $(G,H,N)$ is a Frobenius-Wielandt triple.
\end{corollary}

\begin{proof}
Let $M/N$ be the Frobenius kernel of $G/N$.   We know that $G/N = (M/N)(H/N)$ and $(M/N) \cap (H/N) = N/N$.  It follows that $G = HM$ and $H \cap M = N$.  Let $g \in G \setminus N$.  We know that $gN$ is conjugate to an element of $H/N$.  It follows that $g$ is conjugate to an element of $H$.  Hence, every element of $G \setminus N$ is conjugate to an element of $H$.  By Lemma \ref{conjugacy}, we see that $(G,H,N)$ is a Frobenius-Wielandt triple.	
\end{proof}

\section{Normal subgroups and elements of prime power order}

Using Frobenius-Wielandt triples, we can determine the groups $G$ and primes $p$ with a normal subgroup $N$ so that every element of $G \setminus N$ has $p$-power order.  We now are ready to prove Theorem \ref{main one} from the Introduction which we restate.

\begin{theorem} \label{p-power}
Let $G$ be a group, let $N$ be a normal subgroup, and let $p$ be a prime.  If $P$ is a Sylow $p$-subgroup of $G$, then every element of $G \setminus N$ has $p$-power order if and only if  either (1) $G = P$ or (2) $G = PN$ and $(G,P,P \cap N)$ is a Frobenius-Wielandt triple.
\end{theorem}

\begin{proof}
Suppose first that every element of $G \setminus N$ has $p$-power order.  If $G$ is a $p$-group, then the result is obvious.  Thus, we assume that $G$ is not a $p$-group.  Then $G = PN$.  Notice that all the elements of $G \setminus N$ are conjugate to an element of $P$.  By Lemma \ref{conjugacy}, $(G,P,P \cap N)$ is a Frobenius-Wielandt triple.  Conversely, if $G$ is a $p$-group, then obviously every element in $G \setminus N$ has $p$-power order.  Thus, suppose that $G = PN$ and $(G,P,P \cap N)$ is a Frobenius-Wielandt triple.  Then by Lemma \ref{conjugacy} every element in $G \setminus N$ is conjugate to an element in $P$.  Thus, every element in $G \setminus N$ has $p$-power order. 
\end{proof}

We note that $N$ need not be solvable in Theorem \ref{p-power}.  Let $G$ be the group $M_{10}$ and take $N$ to be the normal subgroup isomorphic to ${\rm PSL} (2,9) \cong A_6$.  One can see that all of the elements in $G \setminus N$ have order $4$ or $8$, and thus, have $2$-power order.  At this time, this essentially is the only example we know of where $N$ is not solvable.  It would be interesting to study the question: suppose $G$ is a group, $N$ is a normal subgroup, $p$ is a prime, $P$ is a Sylow $p$-subgroup so that that $(G,P,P \cap N)$ is a Frobenius-Wielandt group, $O_p (G) = 1$, and $G$ is nonsolvable.  Is this enough to imply that $G \cong M_{10}$?  Or do other examples exist?

Next we look at necessary and sufficient conditions for $N$ normal in $G$, $G \setminus N$ to have all elements having prime power order, and $G/N$ being solvable.   We see that we obtain iterated Frobenius-Wielandt triples.

\begin{theorem} \label{triples}
Let $G$ be a group and let $N$ be a normal subgroup so that $G/N$ is solvable.  Then all elements in $G \setminus N$ have prime power order if and only if one of the following occur:
\begin{enumerate}
\item $G$ is a $p$-group for some prime $p$.
\item There is a prime $p$ and a Sylow $p$-subgroup $P$ so that $G = NP$ and $(G,P,P \cap N)$ is a Frobenius-Wielandt triple.
\item There are primes $p$ and $q$ and Sylow $p$- and $q$-subgroups $P$ and $Q$ respectively, and a normal subgroup $M$ in $G$ so that
      \begin{enumerate}
      	\item $M = NQ$ and $G = MP$.
      	\item $(G,P,P \cap M)$ is a Frobenius-Wielandt triple.
      	\item Either $M = Q$ or $(M,Q,Q \cap N)$ is a Frobenius-Wielandt triple.
      \end{enumerate}
\item There are primes $p$ and $q$ and Sylow $p$- and $q$-subgroups $P$ and $Q$ respectively, and normal subgroups $M$ and $K$ in $G$ so that 
      \begin{enumerate}
      	\item $K = N(K \cap P)$, $M = KQ$, and $G = MP$.
      	\item $(G,P,P \cap M)$ and $(M,Q,Q \cap K)$ are Frobenius-Wielandt triples.
      	\item Either $K \le P$ or $(K,P \cap K,P \cap N)$ is a Frobenius-Wielandt triple.
      \end{enumerate} 
\end{enumerate}
\end{theorem}

\begin{proof}
We first assume that all elements in $G \setminus N$ have prime power order.  This implies that all elements in $G/N$ have prime power order.  We know that $G/N$ is one of the following: (1) a $p$-group for some prime $p$, (2) a Frobenius group whose Frobenius kernel is a $q$-group and whose Frobenius complement is a $p$-group for distinct primes $p$ and $q$, or (3) a $2$-Frobenius group that is a $\{ p, q\}$- group for distinct primes $p$ and $q$.  

If $G$ is a $p$-group, then we have Conclusion (1).  Thus, we may assume that $G$ is not a $p$-group.   Thus, we can set $M$ to be minimal so that $M$ is normal in $G$, $N \le M < G$, and $G/M$ is a $p$-group for some prime $p$.  We know that all elements of $G \setminus M$ are contained in $G \setminus N$ and so have prime power order. Also, since $p$ divides $o(gM)$ for all $g \in G \setminus M$, we see that $p$ divides $o (g)$ for all $g \in G \setminus M$.  We conclude that all elements in $G \setminus M$ have $p$-power order.  By Theorem \ref{conjugacy}, we see that $G = MP$ and $(G,P,P \cap M)$ is a Frobenius-Wielandt triple.  If $M = N$, then we have Conclusion (2).

We assume $N < M$.  We can now set $K$ minimal so that $K$ is normal in $G$, $N \le K < M$, and $M/K$ is a $q$-group.  Note that we must be in the situation where $G/K$ is a Frobenius group.  As in the last paragraph, we can argue that all of the elements in $M \setminus K$ have $q$-power order.  Using Theorem \ref{p-power}, we have either $M$ is a $q$-group or $M = KQ$ and $(M,Q,Q \cap K)$ is a Frobenius-Wielandt triple.  If $K = N$, then we have Conclusion (3).  Note that if $M$ is a $q$-group, then we must have $K = N$ and $M \le Q$.  Since $G/M$ is $p$-group, we conclude that $M = Q$.

Finally, we suppose that $N < K$.  Observe that we are in the situation where $G/N$ is a $2$-Frobenius group.  In this case, we argue as in the previous two paragraphs to see that all elements in $K \setminus N$ have $p$-power order.  By Theorem \ref{p-power}, we see that either $K$ is a $p$-group or  $K = N (K \cap P)$ and $(K, K \cap P, N \cap P)$ is a Frobenius-Wielandt triple. This proves Conclusion (4).

Conversely, suppose (1), (2), (3), or (4) occurs.  Obviously, if $G$ is a $p$-group, then all elements in $G \setminus N$ have prime power order.  In (2), we apply Theorem \ref{p-power} to see that all of the elements in $G \setminus N$ are $p$-powered ordered and so, have prime power order.  In (3), we apply Theorem \ref{p-power} to see that all of the elements in $G \setminus M$ have $p$-power order and the elements in $M \setminus N$ have $q$-power order; so all elements in $G \setminus N$ have prime power order.  In (4), we use Theorem \ref{p-power} to see that all of the elements in $G \setminus M$ and $K \setminus N$ have $p$-power order and in $M \setminus K$ have $q$-power order; so all elements in $G \setminus N$ have prime power order.
\end{proof}

We consider some examples.  In particular, we present examples that are not Frobenius groups or $2$-Frobenius groups.

Suppose $M_1 = A_4 \times Z_9$ and have $\sigma$ be an automorphism of order $2$ so that $\sigma$ acting on $A_4$ is isomorphic to $S_4$ and $\sigma$ is the automorphism of order $2$ on $Z_9$.  Take $G_1$ to be the semi-direct product of $\langle \sigma \rangle$ acting on $M_1$.   We take $K_1$ to be the Sylow $2$-subgroup in $A_4$, so $K_1$ is a Klein $4$-group and take $N_1 = K_1 \times Z_9$.  Take $p = 2$ and $q = 3$.  Write $P_1$ for a Sylow $2$-subgroup of $G_1$ and $Q_1$ for a Sylow $3$-subgroup of $G_1$.  Observe that $P_1 \cap M_1 = K_1$ and $G = N_1 P_1$.  Also, $N_1$ is normal in $G_1$ and $N_1/K_1 \cong S_3$ is a Frobenius group.  By Lemma \ref{conjugacy}, $(G_1,P_1,N_1)$ is a Frobenius-Wielandt triple.   Similarly, $N_1 \cap Q_1 = Z_9$ and $M_1/Z_9 \cong A_4$ is a Frobenius group.  Again, applying Lemma \ref{conjugacy}, $(M_1,Q_1,Z_9)$ is a Frobenius Wielandt triple.  By Theorem \ref{triples} (2), all elements of $G_1 \setminus N_1$ have prime power order.

We now observe that there is a $2$-Frobenius group of the form $(Z_2)^6 \rtimes Z_9 \rtimes Z_2$.  In particular, it can be viewed as the semi-direct product of the Frobenius group $Z_9 \rtimes Z_2$ acting faithfully on $(Z_2)^6$.  Notice that $Z_9 \times Z_2$ acts on $(Z_2)^2$ with $Z_3$ in the kernel of the action.  We define $G_2$ to be the semi-direct product of $Z_9 \rtimes Z_2$ acting on $L_2 = (Z_2)^6 \times (Z_2)^2$ so that the action on $(Z_2)^6$ is faithful and the action on $(Z_2)^2$ has $Z_3$ in the kernel of the action.  We take $L_2 < N_2 < M_2 < G_2$ so that $|G_2:M_2| = 2$ and $|M_2:N_2| = |N_2:L_2| = 3$.  Take $p = 2$ and $q = 3$.  Write $P_2$ for a Sylow $2$-subgroup of $G_2$ and $Q_2$ for a Sylow $3$-subgroup of $G_2$.  Observe that $G = M_2 P_2$ and $M_2 \cap P_2 = L_2$.  Notice that $G/L_2 \cong Z_9 \rtimes Z_2$, thus $(G, P_2, L_2)$ is a Frobenius-Wielandt triple.  Let $C_1 = (Z_2)^2$, $C_2 = (Z_2)^6$, and $Q^*$ be the subgroup of order $3$ in $Q_2$.  Observe that $C_1 = C_{L_2} (Q^*)$ and $N_2 = C_1 \times (C_2 \rtimes Q^*)$.  Notice that $M_2/(Z_2^6 \rtimes Q^*) \cong A_4$, so by Lemma \ref{conjugacy} $(M_2,Q,Q\cap N_2)$ is a Frobenius-Wielandt triple.  Applying Theorem \ref{triples} (2) all elements of $G_2 \setminus N_2$ have prime power order.  

We now obtain restrictions on groups with iterated Frobenius-Wielandt triples.

\begin{theorem} \label{Frobenius}
Let $G$ be a group.  Let $p$ and $q$ be primes so that $P$ and $Q$ are Sylow $p$ and $q$-subgroups, respectively and $M$ and $N$ are normal subgroups so that $G = MP$ and $M = NQ$.  Assume also that $(G,P,P \cap M)$ and $(M,Q,Q \cap N)$ are Frobenius-Wielandt triples.  Then the following are true:
\begin{enumerate}
\item $N_G (Q)$ is a Frobenius group with Frobenius kernel $Q$.  
\item $G/N$ is a Frobenius group with Frobenius kernel $M/N$.
\item If $P$ is chosen so that $P \cap N_G (Q) = N_P (Q)$ is a Sylow $p$-subgroup of $N_G (Q)$, then $N_P (Q)$ Frobenius complement of $N_G (Q)$ and $P = (N \cap P) \rtimes N_P (Q)$.
\item If $O^p (N) < N$, then $G/O^p (N)$ is a $2$-Frobenius group.
\item Either $N_G (Q \cap N) = N_G (Q)$ or $N_G (Q \cap N)/(Q \cap N)$ is a $2$-Frobenius group.
\item $G$ is a $\{ p, q\}$-group.
\end{enumerate}
\end{theorem}

\begin{proof}
Since $(M,Q,Q \cap N)$ is a Frobenius-Wielandt triple, we know that $Q = N_M (Q)$.   This implies that $N_G (Q) \cap M = Q$.  On the other hand, by the Frattini argument, we know that $G = M N_G (Q)$.  We have $|N_G (Q):N_M (G)| = |G:M| = |MP:M| = |P:P \cap M|$.  By Theorem \ref{conjugacy}, all elements in $G \setminus M$ have $p$-power order.  We see that $N_G (Q) \setminus N_M (G) \subseteq G \setminus M$.   Suppose that $P$ is chosen so that $N_P (Q)$ is a Sylow $p$-subgroup of $N_G (Q)$.  Thus, all elements in $N_G (Q) \setminus N_M (Q) = N_G (Q) \setminus Q$ are conjugate to $N_P (Q)$.    We conclude that $N_G (Q)$ is a Frobenius group with Frobenius kernel $Q$ and Frobenius complement $N_P (Q)$.  We see that $M \cap P = N \cap P$ and $P = (N \cap P) \rtimes N_P (Q)$.  Note that $G/N \cong N_G (G)/(Q \cap N)$ is a Frobenius group.  Since $M/N \cong Q/Q \cap N$, we deduce that $M/N$ is the Frobenius kernel.

Suppose that $O^p (N) < N$.  We know that all of the elements in $M/O^p(N) \setminus N/O^p(N)$ have $q$-power order.  Hence, $M/O^p (N)$ is a Frobenius group with Frobenius kernel $N/O^p (N)$.  Since we already know that $G/N$ is a Frobenius group with Frobenius kernel $M/N$, it follows that $G/O^p (N)$ is a $2$-Frobenius group.  

It is not difficult to see that $N_G (Q)$ will normalize $Q \cap N$, so we have $N_G (Q) \le N_G (Q \cap N)$.  Suppose that $N_G (Q) < N_G (Q \cap N)$.  We know that $N_M (Q \cap N) = M \cap N_G (Q \cap N)$.  Hence, $Q = N_M (Q) < N_M (Q \cap N)$.  Thus, we may apply Lemma \ref{frob quo} to see that $N_M (Q \cap N)/(Q \cap N)$ is a Frobenius group.  It is not difficult to see that $N_N (Q \cap N)/(Q \cap N)$ will be the Frobenius kernel.  Observe that $N_G (Q \cap N)/N_N (Q \cap N) \cong G/N$, so $N_G (Q \cap N)/N_N (Q \cap N)$ is a Frobenius group with Frobenius kernel $N_M (Q \cap N)/N_N (Q \cap N)$.  We conclude that $N_G (Q \cap N)/(Q \cap N)$ is a $2$-Frobenius group.

We now work to prove (6).  We continue to assume that $P$ has been replaced by a conjugate so that $P \cap N_G (Q)$ is a Sylow $p$-subgroup of $N_G (Q)$.  We have just shown that $N_G (Q)$ is a Frobenius group with Frobenius kernel $Q$ and Frobenius complement $N_P (Q)$.  Let $K$ be a minimal normal subgroup of $G$ contained in $N$.  We first suppose that $K$ is an elementary abelian $r$-group for some prime $r$.  
We know that $G/N$ is a Frobenius group with Frobenius complement $NP$.  Observe that $KP \le NP < G$.  

Suppose that $M = KQ$.  If $r = p$ or $q$, then this would imply that $M$ and hence $G$ is a $\{ p, q \}$-group and we would have (6).  Thus, when $M = KQ$, we assume that $r \not\in \{ p, q \}$.  Notice that $G/K$ will be a $\{p, q\}$-group in this siutation.  On the other hand, if $KQ < M$, then $(G/K,PK/K,(P \cap M)K/K)$ and $(M/K,QK/K,(Q \cap N)K/K)$ are Frobenius-Wielandt triples.  By induction, $N/K$ and hence, $G/K$ would be a $\{ p, q\}$-group.  Thus, if $r$ is $p$ or $q$, we are done.  Thus, in all cases, we may assume that $r$ is not $p$ or $q$.

We know that $N_G (Q)$ is a Frobenius group with Frobenius complement $N_P (Q)$ and Frobenius kernel $Q$.  Thus, $N_G (Q)$ acts on $K$.  If $C_M (K) \not\le N$, then there exist elements in $KQ \setminus N \subseteq G \setminus N$ which have orders divisible by both $q$ and $r$ which is a contradiction.  Thus, we must have $C_M (K) \le N$.  We now know that $N_G (Q)/C_M(K)$ is a Frobenius group.  By Theorem 15.16 of Isaacs' Character Theory \cite{text}, we see that $C_K (N_P (Q)) > 1$.  This implies that there exist elements in $K N_P (G) \setminus N \subseteq G \setminus N$ that have orders divisible by both $p$ and $r$, a contradicton.  Thus, we must have $r \in \{ p , q \}$.  This proves the theorem when $G$ is solvable.

We now assume that $K$ is not solvable.  In fact, we may assume that $G$ is minimal among nonsolvable examples.  This implies that $K$ is divisible by a prime $r$ that is not $p$ and $q$.  Let $R$ be a Sylow $r$-subgroup of $K$.  We know that $M = K N_M(R)$ and $G = K N_G (R)$.  Since $N_K (R) < R$, we see that $N_G (R) < G$.   We may choose $P$ and $Q$ so that $P \cap N_G (R)$ and $Q \cap N_G (R)$ are Sylow $p$- and Sylow $q$-subgroups of $N_G (R)$, respectively.  We write $N_P (R) = P \cap N_G (R)$ and $N_Q (R) = Q \cap N_G (R)$.  Observe that $|N_G (R):N_M (R)|$ divides $|G:M|$.  This implies $|N_G (R):N_M(R)|$ is a $p$-power, and so, $N_G (R) = N_M (R) N_P (R)$.  Also, $|N_M (R):N_N (R)|$ divides $|M:N|$.  Hence, $|N_M (R):N_N (R)|$ is a $q$-power, and thus, $N_M (R) = N_N (R) N_Q (R)$.   Suppose $x \in N_G (R) \setminus N_P(R)$.  This implies that $x \in G \setminus P$.  We know that $P \cap P^x \le P \cap M$.  Hence, $N_P (R) \cap (N_P (R))^x = N_G (R) \cap P \cap (N_G (R))^x \cap P^x \le N_P (R) \cap M$.   Similarly, if $y \in N_M (R) \setminus N_Q (R)$, then $y \in M \setminus Q$.  WE have $Q \cap Q^y \le Q \cap N$.  Hence, $N_Q (R) \cap (N_Q (R))^y = N_G (R) \cap Q \cap (N_G (R))^y \cap Q^y \le N_G (R) \cap Q \cap N \le N_Q (R) \cap N$.  Hence, $N_G (R)$ satisfies the hypotheses of the theorem.  Also, $|N_G (R)| < |G|$ and $|N_G (R)|$ is divisible by $p$, $q$, and $r$.  Thus, $N_G (R)$ cannot be solvable, and it is a contradiction to the assumption that it is a minimal nonsolvable counterexample.  Hence, we have proved Conclusion 6. 
\end{proof}

The following Corollary is Theorem \ref{main two}.

\begin{corollary}\label{pqfrob}
Suppose $G$ has a normal subgroup $N$ so that every element in $G \setminus N$ has prime power order and the orders of these elements are divisible by the distinct primes $p$ and $q$.  Then $G/N$ is either a Frobenius or a $2$-Frobenius group and $G$ is a $\{ p , q\}$-group.
\end{corollary}

\begin{proof}
If all elements of $G \setminus N$ have prime power order, then all elements of $G/N$ have prime power order.  If $G/N$ were nonsolvable, then $|G/N|$ would be divisible by at least three primes by Theorem \ref{nonsolvable}.  Thus, $G/N$ must be solvable.  Obviously, $G/N$ is not a $p$-group; so from Theorem \ref{solvable}, we see that $G/N$ is either a Frobenius group or a $2$-Frobenius group.
	
In either case, by Theorem \ref{triples}, there is a normal subgroup $M$ and Sylow $p$- and $q$-subgroups $P$ and $Q$ respectively so that $G = MP$, $M = NQ$, $(G,P,P \cap M)$ is a Frobenius-Wielandt triple and either $M = Q$ or $(M,Q,Q \cap N)$ is a Frobenius-Wielandt triple.  If $M = Q$, then since $G = MP = QP$, we see that $G$ is a $\{ p, q\}$-group.  Otherwise, observe that $G$ now satisfies the condictions of Theorem \ref{Frobenius}, and so, $G$ is a $\{p,q\}$-group.
\end{proof}

Using Corollary \ref{pqfrob}, we are to prove following theorem.

\begin{theorem} \label{last}
Let $G$ be a group with a normal subgroup $N$ so that $G/N$ is not solvable.  Then all elements in $G \setminus N$ have prime power order if and only if all elements in $G$ have prime power order.
\end{theorem}

\begin{proof}
If every element in $G$ has prime power order, then every element in $G \setminus N$ has prime power order.  We assume the converse.  We assume that every element in $G \setminus N$ has prime power order.  This implies that $G/N$ is nonsolvable and all elements in $G/N$ have prime power order.  Hence, $G/N$ is one of the groups listed in Theorem \ref{nonsolvable}.  We claim for each of those groups that there exist distinct primes $p_1$ and $p_2$ so that $G/N$ has a Frobenius $\{ 2, p_i \}$-subgroup for each $i$.  Let $F_i/N$ be a Frobenius $\{2, p_i \}$- subgroup of $G/N$.  Notice that $F_i/N$ is solvable, and every element in $F_i \setminus N$ is an element in $G \setminus N$; so every element in $F_i \setminus N$ has prime power order.  By Corollary \ref{pqfrob}, we see that $F_i$ is a $\{ 2, p_i \}$-group.  This implies that $N$ is a $\{ 2, p_i \}$-subgroup for $i = 1, 2$.  The only way this can occur is if $N$ is a $2$-group.  Now, we know every element in $N$ has $2$-power order and every element in $G \setminus N$ has prime power order; so we may conclude that every element of $G$ has prime power order.

To see the claim, observe that ${\rm PSL}_2 (4)$ contains $D_6$ and $D_{10}$, ${\rm PSL}_2 (7)$ contains $D_6$ and $D_{14}$, ${\rm PSL}_2 (8)$ contains $D_{14}$ and $D_{18}$, ${\rm PSL}_2 (9)$ contains $D_{10}$ and $D_{18}$, ${\rm PSL}_2 (17)$ contains $D_{18}$ and $D_{34}$, ${\rm PSL}_3 (4)$ contains many possibilities for example $D_6$ and $D_{10}$, ${\rm Sz} (8)$ contains $D_{14}$, the Frobenius group of order $20$, and the Frobenius group of order $52$, and ${\rm Sz} (32)$ contains $D_{62}$, the Frobenius group of order $100$, and the Frobenius group of order $164$.  For the groups in Theorem \ref{nonsolvable} with a nontrivial Fitting subgroup, we can take the product of the Fitting subgroup with a cyclic subgroup of odd order.  For ${\rm PSL}_2 (4)$, we would have subgroups of order $3$ and $5$; for ${\rm PSL}_2 (8)$, we would have subgroups of order $7$ and $9$; for ${\rm Sz} (8)$, the possible orders are $5$, $7$, or $13$; and for ${\rm SZ} (32)$, the possibilities are $25$, $31$, and $41$.
\end{proof} 

We close by proving Theorem \ref{main three}.

\begin{proof}[Proof of Theorem \ref{main three}]
We know that all the elements of $G \setminus N$ have prime power order, so all the elements of $G/N$ have prime power order.  Since three primes divide the orders of these elements, we know $G/N$ is not solvable by Theorem \ref{solvable}.   Applying Theorem \ref{last}, we see that every element in $G$ has prime power order.  Therefore, $G$ appears in the list in Theorem \ref{nonsolvable}.
\end{proof}

\end{document}